\documentclass[12pt]{amsart}
\usepackage{amssymb,amsmath}
\usepackage[mathscr]{eucal}
\usepackage[english]{babel}
\usepackage[latin1]{inputenc}
\usepackage{color}
\usepackage{amsthm}
\usepackage{mathrsfs}
\usepackage{latexsym}
\usepackage{a4wide}
\usepackage{verbatim}
\usepackage{enumerate}

\def\dom{\mathop{\mathrm{Dom}}\nolimits}
\def\im{\mathop{\mathrm{Im}}\nolimits} 
\def\ker{\mathop{\mathrm{Ker}}\nolimits}
\def\rank{\mathop{\mathrm{rank}}\nolimits}
\def\T{\mathcal{T}}
\def\PT{\mathcal{PT}\!}

\def\S{\mathcal{S}}
\def\sdp{\mathop{\!\rtimes\!}\nolimits}
\def\N{\mathbb N}

\newcommand{\R}{\mathcal R}

\newcommand{\G}{\mathcal G}

\newcommand{\la}[2]{{\kern1pt}^{#1}{\kern-.5pt}{#2}}

\renewcommand{\leq}{\leqslant}
\renewcommand{\geq}{\geqslant}

\newcommand{\comp}[2]{#1^{(#2)}}
\newcommand{\proj}[1]{\overline{#1}}

\newcommand{\zt}{\varnothing}

\newcommand{\act}[2]{{\kern1pt}^#1{\kern-.5pt}#2}

\newtheorem{theorem}{Theorem}[section]

\newtheorem{corollary}[theorem]{Corollary}
\newtheorem{lemma}[theorem]{Lemma}

\newenvironment{proof*}[1]{\begin{trivlist}\item[\hskip%
\labelsep{\bf #1.}]}%
{\qed\rm\end{trivlist}}

\begin{document}

\title{Partial transformation  monoids preserving a uniform partition}

\author{Serena Cical\`o, V\'itor H. Fernandes and Csaba Schneider}

\address[Cical\`o]
{Universit\`a degli Studi di Cagliari, 
Dipartimento di Matematica e Informatica, 
Via Ospedale, 72 - 09124 Cagliari, 
Italy; 
e-mail: cicalo@science.unitn.it}

\address[Fernandes]
{Departamento de Matem\'atica, 
Faculdade de Ci\^encias e Tecnologia, 
Universidade Nova de Lisboa, 
Monte da Caparica, 
2829-516 Caparica, 
Portugal; 
also: 
Centro de \'Algebra da Universidade de Lisboa, 
Av. Prof. Gama Pinto 2, 
1649-003 Lisboa, 
Portugal; 
e-mail: vhf@fct.unl.pt}

\address[Schneider]
{Centro de \'Algebra da Universidade de Lisboa, 
Av. Prof. Gama Pinto 2, 
1649-003 Lisboa, 
Portugal; 
e-mail: csaba.schneider@gmail.com}

\begin{abstract} 
The objective of this paper is to study  the monoid
of all partial transformations of a finite set that preserve a uniform
partition. In addition to proving that 
this monoid is a quotient of 
a wreath product with respect to a congruence relation,
we show that it is generated by 5 generators, we compute its
order and determine a presentation on a minimal generating set.
\end{abstract}



\maketitle

\section{Introduction}\label{intro}

The main objective of this paper is a study
of 
the monoid of all partial transformations of a finite set 
that preserve a uniform 
equivalence relation. 
Given a set $\Omega$, a {partial 
transformation} on $\Omega$ is a map $\alpha:\Gamma\rightarrow \Omega$ 
where $\Gamma\subseteq \Omega$ allowing the case
that $\Gamma=\emptyset$.  
The set of partial transformations of $\Omega$ is a monoid under the 
operation of composition. Given an equivalence relation $E$ on $\Omega$, the 
set of partial transformations that preserve $E$ is closed
under composition, and hence it is a submonoid. In this paper
we will investigate  monoids of partial transformations
preserving a uniform equivalence relation $E$; that is,
an equivalence 
relation with the property that the equivalence classes have the same size. 
For $n\geq 1$, the symbols $\PT_n$ and $\T_n$ denote the monoid of
all partial transformations and the monoid of all transformations
of the set $\underline n=\{1,\ldots,n\}$. The wreath product of two
monoids $T$ and $S$ is denoted by $T\wr S$ (see Section~\ref{notations} for the 
definition). The rank of a monoid is the 
cardinality of a least-order generating set.


The main results of this paper are collected in the following theorem.

\begin{theorem}\label{main}
Let $\Omega$ be a finite set, let $E$ be a 
uniform equivalence
relation on $\Omega$ with $m$ equivalence classes each of which has size $n$,
with $n, m\geq 2$, 
and let $\PT_E$ be the monoid of all partial transformations
on $\Omega$ that preserve $E$. Then 
\begin{enumerate}
\item[(i)] $|\PT_E|=(m(n+1)^n-m+1)^m$;
\item[(ii)] $\rank \PT_E=5$;
\item [(iii)] $\PT_E\cong (\PT_n\wr\T_m)/R$ where $R$ is a congruence relation
on $\PT_n\wr\T_m$ generated by a single pair of elements.
\end{enumerate}
\end{theorem}

Parts~(i) and~(iii) of Theorem~\ref{main} are proved in Section~\ref{wreathsec}, while part~(ii)
is verified in Section~\ref{sectrank}.
In addition to proving the assertions in Theorem~\ref{main}, we determine an
explicit generating set for $\PT_E$ with 5~generators  (Corollary~\ref{5gen})
and a
presentation for $\PT_E$ satisfied by 
this generating set (Corollary~\ref{relcor}).

The rank of a semigroup 
is an important invariant and the determination
of the ranks of  semigroups and monoids has been a focus of research in
semigroup theory. 
Let $\Omega$ be a finite set
with at least 2~elements. 
It is well-known that the full symmetric group  of $\Omega$
is generated by
2~generators. (In the language of semigroup theory
one might say that the rank of the symmetric group is 2, but, 
as the rank has a different meaning in permutation group theory, we
avoid expressing this result this way.) Further, the monoid of all transformations 
and the monoid of all partial transformations of $\Omega$ have
ranks $3$ and $4$,  respectively; see \cite[Theorems~3.1.3 and~3.1.5]{GO}. 
The survey~\cite{Fernandes:2002} collects 
these results and similar ones for other classes of transformation monoids, 
in particular, for monoids of order-preserving transformations and 
for some of their extensions. 

Let  $\T_E$ denote 
the stabilizer of 
a non-trivial uniform equivalence relation $E$ in the monoid 
of transformations of a finite set 
$\Omega$.
Huisheng \cite{Huisheng:2005a} proved  that
$\T_E$ is generated by its group of invertible transformations and two 
additional elements. In other words, the relative rank of $\T_E$ with 
respect to its group of units is~2. The group of units of $\T_E$ is 
isomorphic to a
wreath product of two full symmetric groups and Huisheng observed that such
a group
can be  generated by  4 elements.
This led him to conclude that
$\rank\T_E\leq 6$.  Later,  
Ara\'ujo and the third author \cite{Araujo&Schneider:2009} improved this result 
by showing that the wreath product of two symmetric groups can always be
generated by two elements,
and hence $\rank\T_E=4$. 
Some natural submonoids of $\T_E$, 
such as those of the 
order-preserving or orientation-preserving transformations, 
have  already been considered in 
\cite{Fernandes&Quinteiro:2010b,Huisheng&Dingyu:2005,Huisheng.al:2007}, and 
their ranks were determined by the second author and Quinteiro 
\cite{Fernandes&Quinteiro:2010a,Fernandes&Quinteiro:2011sub}.


The wreath product construction in the class of monoids is one
of the main tools of this paper. The full stabilizer in the symmetric group
of a uniform equivalence relation  is a wreath product of two 
smaller symmetric groups and the analogous result holds 
in the full transformation monoid
(see~\cite[Theorem~2.1]{Araujo&Schneider:2009}). One
would expect that the monoid $\PT_E$ of all partial transformations preserving 
a uniform equivalence relation
$E$ is a wreath product of two  partial transformation monoids, but
this is not the case. Nevertheless, $\PT_E$ can still be described in terms
of wreath products. More precisely, $\PT_E$ is isomorphic to a quotient 
of the wreath product $\PT_n \wr \T_m$ where $n$ is the size of 
an equivalence class and $m$ is the number of classes of $E$.
This quotient is taken modulo a well-described
congruence relation $R$ (see Lemma~\ref{l21} for 
details). As claimed by Theorem~\ref{main}(iii), the 
congruence $R$ is generated by a single pair of elements.

The last part of the paper is concerned with determining a presentation
of the monoid $\PT_E$ defined in the previous paragraph. 
Determining presentations of important semigroups,  monoids, and groups 
has been a mainstream research topic since the end of the 19-th century.
For instance, a presentation for the symmetric group $\S_n$
was given by
Moore \cite{Moore:1897} in 1897.
For the full transformation monoid $\T_n$, a presentation  
was given by A\u{\i}zen\v{s}tat \cite{Aizenstat:1958} in 1958 and, 
some years later, in 1961, Popova \cite{Popova:1961} 
established a presentation for the partial transformation monoid $\PT_n$. 
Over the past decades, several authors determined presentations for many 
other classes of monoids; see \cite{Ruskuc:1995} and the 
survey~\cite{Fernandes:2002} on presentations of monoids 
related to order-preserving transformations. 
Finally
we must mention the recent work of East and his collaborators 
\cite{East:2006,EasdownEastFitzGerald:2008,East:2010a,
East:2010b,East:2011a,East:2011b}.

The paper is organized as follows. 
We set our notation concerning transformations, partial transformations
and wreath products in Section \ref{notations}. 
We establish the relation between $\PT_E$ and wreath products
and prove parts~(i) and~(iii) of Theorem~\ref{main} in Section~\ref{wreathsec}.
Then in Sections~\ref{sectrank}  
we verify Theorem~\ref{main}(ii).
Finally in Section~\ref{pressect} we determine a 
presentation of $\PT_E$ on a minimal generating set.

\section{Notation concerning partial transformations and wreath
products}\label{notations}

In this section we establish the notation concerning partial transformations
and wreath products that will be used throughout the paper.

Let $\Omega$ be a finite set. A {\em partial transformation} of $\Omega$
is a map $\alpha:\Gamma\rightarrow\Omega$ where $\Gamma\subseteq\Omega$.
The set $\Gamma$ is the {\em domain} of
$\alpha$ and is denoted by $\dom\alpha$, while
the {\em image} of $\alpha$ is denoted by
$\im\alpha$. 
If $\Gamma=\Omega$ then $\alpha$ is said to be a {\em transformation}. 
Sometimes we want to emphasize that a partial transformation
{\em is} or {\em is not} 
a transformation and in such cases we write 
{\em full transformation} and {\em strictly partial transformation}, respectively.
We allow the case that $\Gamma=\emptyset$ and the corresponding
partial transformation is denoted by $\zt$. 
In this paper, transformations and
partial transformations act on the right unless it
is explicitly stated otherwise. That is, if $\alpha$ is a partial
transformation on $\Omega$ and $i\in\Omega$, then the image  
of $i$ under $\alpha$ is denoted by $i\alpha$. 
The set
of all partial transformations of $\Omega$ 
is a monoid under composition and is denoted
by $\PT(\Omega)$. Given an equivalence relation $E$ on $\Omega$ 
we say that
a partial transformation $\alpha$ {\em preserves} $E$ if for
all $i, j\in\dom \alpha$ such that
$(i,j)\in E$ we have that $(i\alpha,j\alpha)\in E$. 
The equivalence relation
$E$ is said to be {\em uniform} if its equivalence classes have 
the same size. The equivalence relations $\Omega\times\Omega$ and 
$\{(i,i)\mid i\in\Omega\}$ are preserved by $\PT(\Omega)$, 
and these relations are said to be {\em trivial}. 
Hence an equivalence relation is non-trivial if and only if it has 
at least two classes and at least one class has size 2.
For a natural number $n$, let $\underline n$ denote the set $\{1,\ldots,n\}$. 
The symbols $\PT_n$, $\T_n$, and $\S_n$ denote the monoid of
partial transformations, the monoid of transformations, and the group of 
permutations, respectively, acting
on $\underline n$.

We  represent a transformation $\alpha\in\T_n$ by the list 
$[1\alpha,\ldots,n\alpha]$ of images. 
If $\alpha$ is a partial transformation and
$i\not\in\dom\alpha$ then
we put $\zt$ in the place of $i\alpha$. A permutation is usually 
written as a product of cycles. The identity map in $\T_n$ will be 
written as $1$. 

Let us now  review the concept of wreath products of monoids; 
see~\cite[Section~1.2.2]{strh} for details.
Let $m\in\N$,  let $T$ be a submonoid of $\T_m$, and let 
$S$ be any monoid. 
The \textit{wreath product} of $S$ by the transformation monoid $T$, denoted by $S\wr T$, is the semidirect product $S^m\sdp T$ with 
respect to the left action of 
$T$ on the $m$-fold direct power $S^m$ defined as
\begin{equation}\label{action}
{\kern1pt}^\tau{\kern-.5pt}(s_1,s_2,\ldots,s_m)=(s_{1\tau},s_{2\tau},\ldots,
s_{m\tau}), 
\end{equation}
for all $\tau\in T$ and $s_1,s_2,\ldots,s_m\in S$. 
Therefore, the wreath product $S\wr T$ is the monoid with underlying set 
$S^m\times T$ and multiplication defined by
\begin{multline}\label{product}
(s_1,s_2,\ldots,s_m;\sigma)(t_1,t_2,\ldots,t_m;\tau)= 
((s_1,s_2,\ldots,s_m){\kern1pt}^\sigma{\kern-.5pt}(t_1,t_2,\ldots,t_m);\sigma\tau)= 
\\(s_1t_{1\sigma},s_2t_{2\sigma},\ldots,s_mt_{m\sigma};\sigma\tau),
\end{multline}
for all $(s_1,s_2,\ldots,s_m;\sigma),(t_1,t_2,\ldots,t_m;\tau)\in S^m\times T$. 

An element $\beta$ of a wreath product $S\wr T$ 
can be written uniquely in the form $(\comp \beta 1,\ldots,
\comp \beta m;\proj \beta)$ where 
$\comp \beta 1,\ldots,\comp \beta m\in S$ and $\proj \beta\in T$. 
In the rest 
of the paper, for an element $\beta\in S\wr T$, these components of $\beta$
will be denoted by $\comp \beta 1,\ldots,\comp \beta m,\proj \beta$.

\section{Wreath products and  partial endomorphisms of a 
uniform partition}\label{wreathsec}

Let $m, n\in\N$ and set $\Omega=\underline n\times\underline m$. 
Let $E$ denote the equivalence relation on $\Omega$ that is defined
by the rule that $((i,j),(k,l))\in E$ if and only if $j=l$. 
Then $E$ has $m$ equivalence classes each of which has size $n$, and hence $E$
is uniform.
Let $\PT_{n\times m}$, $\T_{n\times m}$, and $\S_{n\times m}$ denote the monoid
of  all partial transformations, the monoid of all 
transformations, and the group
of all permutations on $\Omega$ 
preserving the equivalence relation $E$. 
An 
isomorphism between $\T_{n\times m}$ and the wreath product 
$\T_n\wr\T_m$ can easily be constructed as follows
(see~\cite[Lemma~2.1]{Araujo&Schneider:2009}).  Let 
$\alpha\in \T_{n\times m}$. The transformation
$\alpha$ induces a transformation $\proj \alpha$ 
on the set of  equivalence classes of $E$.
Since these  equivalence classes 
are indexed by the elements of $\underline m$, the transformation
$\overline\alpha$ can be viewed as an element of $\T_m$. 
Further, for each $j\in\underline m$,
we can define a transformation $\comp\alpha j$ on $\underline n$ as follows:
$i\comp \alpha j=k$ if $(i,j)\alpha=(k,j\overline\alpha)$.
It is proved in \cite[Lemma 2.1]{Araujo&Schneider:2009} that the map
$\alpha\mapsto (\comp\alpha 1,\ldots,\comp \alpha m;\proj\alpha)$ is 
an isomorphism between the monoids $\T_{n\times m}$ and $\T_n\wr\T_m$.



One would think that the monoids $\PT_{n\times m}$ may be isomorphic to 
a wreath product constructed from  $\PT_n$ and $\T_m$, 
but unfortunately this is not the case.
Nonetheless, the monoid $\PT_{n\times m}$ can be described as a 
quotient of the wreath product $\PT_n\wr\T_m$ as follows. 
We define a homomorphism $\varphi$ 
from $\PT_n\wr\T_m$ to $\PT_{n\times m}$ by defining
a ``partial action'' of $\PT_n\wr\T_m$ on 
$\Omega=\underline n\times\underline m$. Let  
$\alpha\in  \PT_n\wr\T_m$. We define $\alpha\varphi\in\PT_{n\times m}$ 
as follows. Let 
\begin{equation}\label{domain}
\dom\alpha\varphi=\{(i,j)\mid i\in\dom \comp \alpha j\}
\end{equation}
and for $(i,j)\in\dom\alpha\varphi$ we set $(i,j)(\alpha\varphi)=
(i\comp\alpha j,j\overline\alpha)$. 
It is routine calculation to check that $\varphi:\PT_n\wr\T_m\rightarrow
\PT_{n\times m}$ is a homomorphism.

For $j\in\underline m$, let 
$\comp\varepsilon j$ denote the homomorphism 
$\PT_n\rightarrow \PT_n\wr\T_m$ defined by the rule
 $\alpha\mapsto(1,\ldots,1,\alpha,1,\ldots,1;1)$
where the non-trivial factor of the image appears in the $j$-th position. 
Further, let $\proj\varepsilon$ denote the homomorphism $\T_m\rightarrow \PT_n\wr\T_m$ 
mapping $\alpha\mapsto(1,\ldots,1;\alpha)$. The homomorphisms
$\comp\varepsilon j$ and
$\proj\varepsilon$ are clearly injective.

\begin{lemma}\label{l21}
(i) The homomorphism $\varphi$ defined above is surjective. 

(ii) For $\alpha,\beta \in\PT_n\wr\T_m$, we have that  
$\alpha\varphi=\beta\varphi$ if and only if
$\comp \alpha j=\comp \beta j$ 
for all $j\in\underline m$ and $j\proj\alpha=j\proj\beta$ whenever 
$\comp \alpha j\neq \zt$.

(iii) Let $\tau$ denote the element $[2,2,3,\ldots,m]\proj\varepsilon$ of
$\PT_n\wr\T_m$. 
Then the congruence relation $\ker\varphi$ is generated by  the pair
$(\zt\comp\varepsilon 1,(\zt\comp\varepsilon 1)\tau)$.
\end{lemma}
\begin{proof}
(i) 
Let $\alpha\in\PT_{n\times m}$. 
We  define $\comp\alpha 1,\ldots,\comp\alpha m\in\PT_n$ and 
$\proj\alpha\in
\T_m$ similarly 
as in the argument that $\T_n\wr\T_m\cong \T_{n\times m}$ before the lemma. 
The only difference is that, for some $j$, 
the restriction of $\alpha$ to the equivalence class
$\{(i,j)\mid i\in \underline n\}$ may be equal to $\zt$. In this case
we let $j\proj\alpha=1$.
Then
it is clear that $(\comp\alpha 1,\ldots,\comp \alpha m;\proj\alpha)\varphi=
\alpha$, and hence $\varphi$ is surjective.

(ii) 
Let us first assume  that $\alpha,\beta\in\PT_n\wr\T_m$ such that 
$\alpha\varphi=\beta\varphi$. 
If $\comp\alpha j\neq\comp \beta j$, for some $j$,  
then there is $i\in\underline n$ such that
$i\comp \alpha j\neq i\comp \beta j$ 
(including the possibility that either $i\not\in\dom\alpha$ 
or $i\not\in\dom\beta$). Then $(i,j)(\alpha\varphi)\neq(i,j)(\beta\varphi)$, 
which is a contradiction. 
Similarly, if $j\in\underline m$ is such that $\comp\alpha j\neq \zt$, but 
$j\proj\alpha\neq j\proj\beta$, then we have, for $i\in\dom\comp\alpha j$, 
that
$(i,j)(\alpha\varphi)\neq(i,j)(\beta\varphi)$.
Therefore we obtain a contradiction again. Thus one direction of assertion~(ii)
is valid.

To show the other direction, 
let us now suppose  that $\alpha, \beta\in\PT_n\wr\T_m$ such that
 $\comp \alpha j=\comp \beta j$ 
for all $j\in\underline m$ and $j\proj\alpha=j\proj\beta$ whenever 
$\comp \alpha j\neq \zt$. Equation~\eqref{domain} implies that
$\dom(\alpha\varphi)=\dom(\beta\varphi)$.  Let $(i,j)\in\dom(\alpha\varphi)$. 
Then $\comp\alpha j\neq \zt$ and $\comp\beta j\neq \zt$ and so
$$
(i,j)(\alpha\varphi)=(i\comp \alpha j,j\proj\alpha)=(i\comp\beta j,j\proj\beta)=(i,j)(\beta\varphi).
$$
Hence $\alpha\varphi=\beta\varphi$ as required.

(iii) Let $Q$ be the smallest congruence relation on 
the monoid $\PT_n\wr\T_m$ 
such that
$(\zt\comp\varepsilon 1,(\zt\comp\varepsilon 1)\tau)\in Q$.
By the description of $\ker\varphi$ in statement~(ii), we have that
$(\zt\comp\varepsilon 1,(\zt\comp\varepsilon 1)\tau)\in \ker\varphi$ 
and hence $Q\subseteq \ker\varphi$. It remains to show
that $\ker\varphi\subseteq Q$.

For $i, s\in\underline m$, 
let $\tau_{i,s}$ denote the transformation $[1,2,\ldots,i-1,s,
i+1,i+2,\ldots,m]$. We claim that
\begin{equation}\label{qeq}
(\zt\comp\varepsilon i,(\zt\comp\varepsilon i)(\tau_{i,s}\proj\varepsilon))\in Q\quad
\mbox{for all}\quad i, s\in\underline m.
\end{equation}
First we verify this claim in the case when $i=1$.  Let $s\in\{2,\ldots,m\}$. 
As
$$
((2\,s)\proj\varepsilon) (\zt\comp\varepsilon 1)((2\,s)\proj\varepsilon)= \zt\comp\varepsilon 1
$$
while
$$
((2\,s)\proj\varepsilon) (\zt\comp\varepsilon 1)\tau((2\,s)\proj\varepsilon)= 
(\zt\comp\varepsilon 1)(\tau_{1,s}\proj\varepsilon),
$$
we find that 
$(\zt\comp\varepsilon 1,(\zt\comp\varepsilon 1)(\tau_{1,s}\proj\varepsilon))\in Q$.
Now conjugating
this pair with $(1\,i)\proj\varepsilon$ gives~\eqref{qeq}. 

Suppose that $\alpha, \beta\in \PT_n\wr\T_m$ such that 
$(\alpha,\beta)\in \ker\varphi$. We are required to show that $(\alpha,\beta)\in Q$. 
We have, for all $j\in\underline m$, that 
$\comp\alpha j=\comp\beta j$ and that $j\overline\alpha=j\overline\beta$ whenever
$\comp\alpha j\neq\zt$. 
Let $I=\{j\mid j\in\underline m\mbox{ with } j\overline\alpha\neq j\overline\beta\}$. 
We show by induction
on $|I|$ that $(\alpha,\beta)\in Q$. 
If $I=\emptyset$, then 
$\alpha=\beta$, and in this case the claim is clearly valid. 

Suppose that $I\neq\emptyset$ and 
that $(\alpha_1,\beta_1)\in \ker\varphi$ implies that $(\alpha_1,\beta_1)\in Q$ 
whenever $|\{j\mid j\overline\alpha_1\neq j\overline\beta_1\}|\leq |I|-1$. 
Let $j_0$ be the minimal element of $I$. After possibly replacing
$\alpha$ and $\beta$ with $((1\,j_0)\proj\varepsilon)\alpha$
and $((1\,j_0)\proj\varepsilon)\beta$, we may assume that $j_0=1$. 
Thus we  have that $1\overline\alpha=r$ and $1\overline\beta=s$
with $r\neq s$ and that 
$\comp\alpha 1=\comp\beta 1=\zt$. 
Suppose first that $s\in\im\proj\alpha$; that is, 
$s=u\overline\alpha$ with some $u\in\underline m$. Then
$$
(\zt\comp\varepsilon 1)\alpha=\alpha
$$
while
$$
(\zt\comp\varepsilon 1)(\tau_{1,u}\proj\varepsilon)\alpha=
(\comp\alpha 1,\comp\alpha 2,\ldots,\comp\alpha m;\overline\alpha')
$$
where $1\overline\alpha'=s$ and $j\overline\alpha'=j\overline\alpha$ 
for all $j\in\{2,\ldots,m\}$. 
Hence
 $(\alpha, (\comp\alpha 1,\comp\alpha 2,\ldots,\comp\alpha m;\overline\alpha'))\in Q$. 
Since $|\{j\mid j\overline\alpha'\neq j\overline\beta\}|\leq |I|-1$, 
we obtain from the induction hypothesis that 
$(\beta,(\comp\alpha 1,\comp\alpha 2,\ldots,\comp\alpha m;\overline\alpha'))\in Q$. 
Thus $(\alpha,\beta)\in Q$ as claimed.

If $s\not\in\im\overline\alpha$, but $r\in\im\overline\beta$ then, interchanging
the role of $\alpha$ and $\beta$,  the same
argument shows that $(\alpha,\beta)\in Q$.

Suppose now  that $s\not\in\im\overline\alpha$ and $r\not\in\im\overline\beta$. 
If there is some $j\in\underline m$ such that $\comp\alpha j\neq \zt$ then
$j\overline\alpha=j\overline\beta$. Set $u=j\overline\alpha$ and let 
$\overline\alpha'=[u,2\overline\alpha,3\overline\alpha,\ldots,m\overline\alpha]$
and $\overline\beta'=
[u,2\overline\beta,3\overline\beta,\ldots,m\overline\beta]$.
By the argument above, we have that
$(\alpha,(\comp\alpha 1,\ldots,\comp\alpha m;\overline\alpha'))\in Q$
and $(\beta,(\comp\beta 1,\ldots,\comp\beta m;\overline\beta'))\in Q$.
By the induction hypothesis, $((\comp\alpha 1,\ldots,\comp\alpha m;\overline\alpha'),(\comp\beta 1,\ldots,\comp\beta m;\overline\beta'))\in Q$
which implies that $(\alpha,\beta)\in Q$ in this case.

Finally we prove the claim when $\comp\alpha j=\zt$ for all $j\in
\underline m$. 
In this case we are required to show that
$((\zt,\ldots,\zt;\overline\alpha),(\zt,\ldots,\zt;\overline\beta))\in Q$ 
for all $\overline\alpha, \overline\beta\in\T_m$.
The elements $\alpha$ with the property that $\comp\alpha j=\zt$ for 
all $j\in\underline m$ form a submonoid $T$ in $\PT_n\wr \T_m$ that is
isomorphic to $\T_m$ via the isomorphism $\alpha\mapsto\proj\alpha$. 
The relation $Q_T=Q\cap (T\times T)$ is a congruence relation on $T$. 
Equation~\eqref{qeq} implies that
$$
\left(\prod_{j\in\underline m}\zt\comp\varepsilon j,\prod_{j\in\underline m}(\zt\comp
\varepsilon j)({\tau_{j,1}}\proj\varepsilon)\right)\in 
Q_T. 
$$
Now $\prod_{j\in\underline m}\zt \comp\varepsilon j=(\zt,\ldots,\zt;1)$ while
$\prod_{j\in\underline m}(\zt\comp\varepsilon j)({\tau_{j,1}}\proj\varepsilon)=
(\zt,\ldots,\zt;[1,\ldots,1])$.
Therefore 
$((\zt,\ldots,\zt;1),(\zt,\ldots,\zt;[1,\ldots,1]))\in Q_T$. 
Now~\cite[Theorem~6.3.10]{GO} gives  that $Q_T$ must be the universal 
congruence which shows that 
$((\zt,\ldots,\zt;\overline\alpha),(\zt,\ldots,\zt;\overline\beta))\in Q_T$ 
for all $\overline\alpha, \overline\beta\in\T_m$.
Therefore 
$((\zt,\ldots,\zt;\overline\alpha),(\zt,\ldots,\zt;\overline\beta))\in Q$
for all $\overline\alpha, \overline\beta\in\T_m$ as required.
\end{proof}

Part~(iii) of Lemma~\ref{l21} implies Theorem~\ref{main}(iii).

Next we determine the order of $\PT_{n\times m}$ by counting the 
equivalence classes of the relation $\ker\varphi$ 
defined in Lemma~\ref{l21}. Let
$\alpha\in\PT_n\wr\T_m$. We say that
$\alpha$ is in {\em canonical form} if 
$\comp\alpha j=\zt$ implies that $j\proj\alpha=1$ for all
$j\in\underline m$. 

\begin{lemma}\label{ordlemma}
Every equivalence class of $\ker\varphi$ contains precisely
one element in canonical form. Consequently, 
$$
|\PT_{n\times m}|=(m(n+1)^n-m+1)^m.
$$
\end{lemma}
\begin{proof}
Let $\alpha\in\PT_n\wr\T_m$. Then $\alpha=
(\comp\alpha 1,\ldots,\comp\alpha m;\proj\alpha)$ 
where $\comp\alpha 1,\ldots,\comp\alpha m
\in\PT_n$ and $\proj\alpha=\T_m$. 
For $j\in\underline m$, set $\comp \beta j=\comp\alpha j$ and 
define $\proj\beta\in\T_m$ by
the rule that $j\proj\beta=j\proj\alpha$ if $\comp\alpha j\neq \zt$ and
$j\proj\beta=1$ otherwise. Then the element $\beta=(\comp\beta 1,\ldots,\comp\beta m;\proj\beta)$ is in canonical form
and $\alpha\varphi=\beta\varphi$. Hence every congruence class of 
$\ker\varphi$ contains an element in canonical form.

Assume that $\alpha, \beta\in\PT_n\wr\T_m$ are in canonical form
such that $\alpha\varphi=\beta\varphi$. We have by Lemma~\ref{l21} that
$\comp\alpha j=\comp\beta j$ for all $j\in\underline m$ and that
$j\proj\alpha=j\proj\beta$ for all $j\in\underline m$
such that $\comp\alpha j\neq \zt$. On the other hand, if $\comp\alpha j=\zt$ then,
as $\alpha$ and $\beta$ are in canonical form, 
$j\proj\alpha=j\proj\beta=1$. Hence we obtain that $\alpha=\beta$ and
so each of the congruence classes contains precisely one element in 
canonical form.

To compute the order of $\PT_{n\times m}$, let us compute the number of elements
in canonical form. For $r\in\underline m$, 
let $A_r$ denote the number of elements $\alpha$ in canonical form
such that $|\{j\in\underline m\mid \comp\alpha j\neq \zt\}|=r$. Then
$$
A_r=\binom{m}r\left(|\PT_n|-1\right)^rm^r.
$$
Thus 
$$
|\PT_{n\times m}|=\sum_{r=0}^m\binom{m}r\left(|\PT_n|-1\right)^rm^r=
\sum_{r=0}^{m}\binom{m}{r}(m(n+1)^n-m)^r= (m(n+1)^n-m+1)^m,
$$
as claimed.
\end{proof}

The second assertion of Lemma~\ref{ordlemma} implies
Theorem~\ref{main}(i). 
The table below shows 
the order of the monoid $\PT_{n \times m}$ for small values of $m$ and $n$.  

{\small $$
\begin{array}{r|rrrrrr}
m&|\PT_m|& |\PT_{2\times m}|&|\PT_{3\times m}|&|\PT_{4\times m}|&|\PT_{5\times m}|\\
\hline
1&   2&         9&   64&   625& 7776\\
2&   9&       289& 16129& 1560001&115856201\\
3&   64&     15625& 6859000& 6570725617 &3150905752576\\
4&   625&   1185921& 4097152081& 38875337230081&296120751810639601\\
5&   7776& 115856201& 3150905752576& 296120751810639601&88798957515761812069376
\end{array}
$$}

\section{The rank}
\label{sectrank}


In this section we determine the ranks of $\PT_{n\times m}$ and $\PT_n\wr\T_m$
defined in Section~\ref{wreathsec}.

\begin{theorem}\label{rankmain}
If $m\ge2$ and $n\ge2$, then 
$\rank\PT_n\wr\T_m=\rank\PT_{n\times m}=5$. 
\end{theorem}

Before proving the theorem, we review
some known facts about the ranks and generating sets of the monoids
$\PT_n$, $\T_n$, $\T_n\wr\T_m$ and fix some notation. 
Let us consider the following elements of $\T_n\wr\T_m$: 
\begin{align*}
\xi_1&=\left\{\
\begin{array}{ll}
((1\;2)\comp\varepsilon 2)((2\;3\cdots m)\proj\varepsilon) & \mbox{if $m$ and $n$ are both even}\\
((1\;2)\comp\varepsilon 2)((1\;2\cdots m)\proj\varepsilon) & \mbox{otherwise;}
\end{array}\right.\\
\xi_2&=((1\;2\cdots n)\comp\varepsilon 1)((1\;2)\proj\varepsilon);\\
\tau&=
[2,2,3,\ldots,n]\comp\varepsilon 1;\\
\overline\tau&=[2,2,3,\ldots,m]\proj\varepsilon.
\end{align*} 
Then the following lemma is valid.

\begin{lemma}\label{ranklemma}
\begin{enumerate}
\item[(i)] If $\gamma\in\PT_n$ then $\PT_n=\left<\T_n,\gamma\right>$ if 
and only if $\gamma\in\PT_n\setminus\T_n$ and $\rank\gamma=n-1$. 
\item[(ii)] We have that $\rank\S_n\wr\S_m=2$ and $\rank\T_n\wr\T_m=4$. 
Further,
$\S_n\wr\S_m=\left<\xi_1,\xi_2\right>$ and 
$\T_n\wr\T_m=\left<\xi_1,\xi_2,\tau,\overline\tau\right>$. 
\item[(iii)] Let $P$ be a monoid of partial transformations of a 
finite set and let $G$ be the  group of permutations in 
$P$. If $X$ is a generating 
set of $P$ then $X\cap G$ is a generating set of $G$.
\end{enumerate}
\end{lemma}
\begin{proof}
Statement~(i)  follows from~\cite[Theorem~3.1.5]{GO}. 
Combining~\cite[Theorem~4.1, Lemma~3.2, and 
Theorem~1.1]{Araujo&Schneider:2009} gives statement~(ii). 
(Note that the generators that correspond to $\tau$ and $\overline\tau$ appear 
in~\cite{Araujo&Schneider:2009} in a slightly different form.)
Finally statement~(iii) can be easily proved using 
the argument in~\cite[Lemma~3.1]{Araujo&Schneider:2009}.
\end{proof}


Now we can prove Theorem~\ref{rankmain}. 

\begin{proof}[Proof of Theorem~$\ref{rankmain}$]
First we show that the monoids $\PT_{n\times m}$ and 
$\PT_n\wr\T_m$ are generated by 5 elements. 
Since, by Lemma~\ref{l21}, $\PT_{n\times m}$ is a homomorphic image of 
$\PT_n\wr\T_m$, it suffices to show that $\PT_n\wr\T_m$ is generated by 5 elements.
Let $\sigma_1=[\zt,2,3,\ldots,n]$; then, by Lemma~\ref{ranklemma}(i), $\PT_n=\left<\T_n,\sigma_1\right>$. Set
$\sigma=\sigma_1\comp\varepsilon 1=(\sigma_1,1,\ldots,1;1)$. 
We claim that $\PT_n\wr\T_m=\left<\T_n\wr\T_m,\sigma\right>$. 
Once this assertion is proved, the claim that $\PT_n\wr\T_m$ is generated by 5~elements
will follow, as $\rank \T_n\wr\T_m=4$ (Lemma~\ref{ranklemma}(ii)).

Let us now prove that $\PT_n\wr\T_m=\left<\T_n\wr\T_m,\sigma\right>$. 
Set $W=\left<\T_n\wr\T_m,\sigma\right>$. 
First we note that $\T_m\proj\varepsilon\leq W$. The fact 
that $\PT_n=\left<\T_n,\sigma_1\right>$ implies that 
$\PT_n\comp\varepsilon 1=\left<\T_n\comp\varepsilon 1,\sigma\right>$. As $\T_n\comp\varepsilon 1\leq W$, we obtain that
$\PT_n\comp\varepsilon 1\leq W$. 
For $j\geq 1$, it is a consequence of equation~\eqref{product} that 
$((1\,j)\proj\varepsilon)(\PT_n\comp\varepsilon 1)((1\,j)
\proj\varepsilon)=\PT_n\comp\varepsilon j$.
As $\PT_n\comp\varepsilon 1\leq W$ and
$\S_m\proj\varepsilon\leq W$,
we find that $\PT_n\comp\varepsilon j\leq W$ for all 
$j\in\underline m$.  Since the monoids $\T_m\proj\varepsilon$ and 
$\PT_n\comp\varepsilon j$ with $j\in\underline m$ are contained in $W$ and they 
generate $\PT_n\wr\T_m$, we obtain that $\PT_n\wr\T_m\leq W$. As clearly, $W\leq \PT_n\wr\T_m$ we
conclude that $\PT_n\wr\T_m=W$, as claimed.

Next we show that the monoids $\PT_{n\times m}$ and 
$\PT_n\wr\T_m$ cannot be generated by fewer than~5 elements. 
We note, as above, that $\PT_{n\times m}$ is a homomorphic image of 
$\PT_n\wr\T_m$, and we need only prove this for
$\PT_{n\times m}$.
Let $G$ denote the unit group of $\PT_{n\times m}$. 
Let $X$ be a subset of $\PT_{n\times m}$ such that $\PT_{n\times m}=\left<G,X\right>$. 
The group $G$, being isomorphic to 
$\S_n\wr\S_m$, is generated by 2~elements
(Lemma~\ref{ranklemma}(ii)), 
and hence, by Lemma~\ref{ranklemma}(iii),
it suffices to show that $|X|\geq 3$. 
Since
 the rank of a  product is not greater than the
minimum of the ranks of its factors and elements
of rank $nm$ form a subgroup,
we obtain that $X$ must contain at 
least one element of rank $nm-1$. 

We claim that $X$ must in fact 
contain at least two elements of rank $nm-1$, namely
one strictly partial transformation, and one full transformation. Assume 
by contradiction that  $X$ only contains one element, $\alpha$ say, 
with rank $nm-1$. Then an element $\beta\in\PT_{n\times m}$  with rank $nm-1$ 
can be written as a product in $X$ and $\S_n\wr\S_m$ and such a product 
cannot contain a transformation with rank less than $nm-1$. Hence 
$\beta$ must be written as a product in $\S_n\wr\S_m$ and $\alpha$. 
Now, if $\alpha$ were a strictly partial transformation, then
no full transformation $\beta$ with  rank $nm-1$ could be written in this form,
which is a contradiction. Similarly, if $\alpha$ were a full transformation, then
no strictly partial transformation with rank $nm-1$ 
could be written as a product in $\S_n\wr\S_m$ and $\alpha$.
Thus $X$ must contain at least two elements with rank $nm-1$, 
namely one full and one strictly partial, as claimed.

Finally we prove that $X$ must contain a third element.  Assume by
contradiction that $X=\left<\alpha_1,\alpha_2\right>$. We have,
by the previous paragraph, that $\rank\alpha_1=\rank\alpha_2=nm-1$ and we may 
assume without loss of generality that $\alpha_1$ is full and $\alpha_2$ is 
strictly partial. 
Let $\beta_1\in\T_m$ be an element of rank $m-1$ and set
$\beta=\beta_1\proj\varepsilon\varphi$
where $\varphi:\PT_n\wr\T_m\rightarrow \PT_{n\times m}$ is the homomorphism defined before 
Lemma~\ref{l21}. 
By Lemma~\ref{l21}, the elements $\alpha_1$, 
$\alpha_2$, and $\beta$ have unique preimages $\alpha_1'$, $\alpha_2'$, and 
$\beta'$ in $\PT_n\wr\T_m$, respectively, under the homomorphism
$\varphi$. Since $\beta$ must be written as a word in $G$, $\alpha_1$, 
and $\alpha_2$, 
we have that $\beta'$ must be written as a word in $\S_n\wr\S_m$, $\alpha_1'$ and 
$\alpha_2'$. However, the map $\alpha\mapsto\overline\alpha$, 
defined before Lemma~\ref{l21}, is a homomorphism
$\PT_n\wr\T_m\rightarrow \T_m$ by equation~\eqref{product}. As $\overline{\alpha_1'}, 
\overline{\alpha_2'}\in\S_m$, we obtain that $\overline{\beta'}\in\S_m$.
Since $\overline\beta'=\beta_1$ and, by assumption, $\rank\beta_1=m-1$, this 
is a contradiction. Thus $X$ must contain a third element as claimed.
\end{proof}

The last lemma implies Theorem~\ref{main}(ii).
Lemma~\ref{ranklemma}(ii) and the argument of the proof of 
Theorem~\ref{rankmain} yield explicit generating sets for $\PT_n\wr\T_m$ and $\PT_{n\times m}$. 
Recall that the homomorphism 
$\varphi:\PT_n\wr\T_m\rightarrow \PT_{n\times m}$ was defined before 
Lemma~\ref{l21}.

\begin{corollary}\label{5gen}
Let $\xi_1$, $\xi_2$, $\tau$, $\overline\tau$ be the elements of $\T_n\wr\T_m$ defined before Lemma~$\ref{ranklemma}$ and let $\sigma$ be the element
defined in the proof of Theorem~$\ref{rankmain}$. 
Then $\{\xi_1,\xi_2,\tau,\overline\tau,\sigma\}$ and $\{\xi_1\varphi,\xi_2\varphi,\tau\varphi,\overline\tau\varphi,\sigma\varphi\}$ 
are minimal generating sets for 
 $\PT_n\wr\T_m$ and $\PT_{m\times n}$ respectively. 
\end{corollary}

\section{Presentations for $\PT_n\wr\T_m$ and $\PT_{n\times m}$}\label{pressect}

Let $S$ be a monoid with generating set $X$ and let $X^*$ denote the free
monoid on $X$. 
A word $w$ in $X$ can be considered as an element of $S$ and
also as an element of $X^*$, but this will cause no 
confusion. Let $R$ denote the  equivalence relation
on $X^*$ defined by the rule that $(w_1,w_2)\in R$ if and only if 
$w_1=w_2$ holds in $S$. Then $R$ is in fact 
a congruence relation. Suppose, further,
that $\R\subseteq X^*\times X^*$ such that the smallest congruence
relation on $X^*$ containing $\R$ is $R$. In this case we say that 
$\R$ is a {\em set of  defining relations} for $S$ with 
respect to the generating set $X$ and write that
$S=\left<X\mid\R\right>$. The pair $\left<X\mid\R\right>$ is 
called a {\em presentation} for $S$. A relation $(w_1,w_2)\in\R$ is
usually written as $w_1=w_2$.

Let $n, m\geq 2$. 
To discuss presentations of $\PT_n\wr\T_m$ and $\PT_{n\times m}$ let us
fix generators of these monoids.
As in Section~\ref{wreathsec}, for $i\in\underline m$, 
$\comp\varepsilon i$ denotes the $i$-th coordinate
embedding $\PT_n\rightarrow 
\PT_n\wr\T_m$ and  $\proj\varepsilon$ denotes the embedding $\T_m\rightarrow 
\PT_n\wr\T_m$ into the $(m+1)$-th component. 
Let us consider the following elements of
$\PT_n\wr\T_m$:
\begin{align*}
\pi&=(1\,2)\comp\varepsilon 1;\\
\varrho&=(1\,2\ldots n)\comp\varepsilon 1;\\
\tau &=[2,2,3,4,\ldots,n]\comp\varepsilon 1;\\
\sigma &=[\zt,2,3,\ldots,n]\comp\varepsilon 1;\\
\overline\pi&=(1\,2)\proj\varepsilon;\\
\overline\varrho&=(1\,2\ldots m)\proj\varepsilon;\\
\overline\tau&=[2,2,3,4,\ldots,m]\proj\varepsilon.\\
\end{align*}

We have that $\left<\pi,\varrho,\tau,\sigma\right>\cong
\PT_n$ under the map $\comp\varepsilon 1$ and that 
$\left<\overline\pi,\overline\varrho,\overline\tau\right>\cong\T_m$ under the embedding $\proj\varepsilon$. 
Hence we may identify $\PT_n$ with $\left<\pi,\varrho,\tau,\sigma\right>$ 
and $\T_m$ with $\left<\overline\pi,\overline\varrho,\overline\tau\right>$.
Let $\R_P$ and $\R_T$ be sets of defining relations for 
$\PT_n$ and $\T_m$ with respect to the 
generating sets $\{\pi,\varrho,\tau,\sigma\}$ and 
$\{\overline\pi,\overline\varrho,\overline\tau\}$, respectively.
Explicit expressions for $\R_T$ and $\R_P$ can be 
found in~\cite[Section~2]{Fernandes:2002}.

Let $S$ and $Q$ be two monoids defined by the presentations 
$\langle X_S\mid \R_S\rangle$ and $\langle X_Q\mid \R_Q\rangle$, 
respectively. 
Let $S\sdp Q$ be a semidirect product of $S$ by $Q$ with respect
to a left action of $Q$ on $S$. 
For each $a\in X_S$ and $b\in X_Q$, denote by ${\kern1pt}^b{\kern-.5pt}a$ a 
(fixed) word of $X_S^*$ that represents the image of $a$ under $b$
in $S$. 
Then Lavers proved in \cite[Corollary 2]{Lavers:1998} 
(see also \cite{Araujo:2001}) that 
\begin{equation}\label{semidir}
S\rtimes Q=\langle X_S\cup X_Q\mid \R_S\cup\R_Q\cup\{ba=({\kern1pt}^b{\kern-.5pt}a)b\mid a\in X_S,b\in X_Q\} \rangle.
\end{equation}
Using this result in the case when the left action is trivial, we obtain
that 
\begin{equation}\label{dirpr}
S\times Q=\langle X_S\cup X_Q\mid \R_S\cup \R_Q\cup\{ ba=ab\mid a\in X_S,b\in X_Q\} \rangle.
\end{equation}

Next define the following sets of relations in 
$\left\{\pi,\varrho,\tau,\sigma,\overline \pi,\overline \varrho,\overline \tau\right\}$:
\begin{align*}
\R_1 =&\{ \overline \varrho^{m-j+1} u \overline \varrho^{m+j-k} v \overline \varrho^{k-1} = 
   \overline \varrho^{m-k+1} v \overline \varrho^{m+k-j} u \overline \varrho^{j-1}\mid
   j, k\in\underline m,\ j<k,\  u, v\in\{\pi,\varrho,\tau,\sigma\}\};\\
\R_2=&\{\overline \pi \overline \varrho^{m-1} u \overline \varrho = u \overline \pi, 
\overline \pi \overline \varrho^{m-j+1} u \overline \varrho^{j-1} = \overline \varrho^{m-j+1} u \overline \varrho^{j-1} \overline \pi\mid 3\leq
j\leq m,\ u\in\{\pi,\varrho,\tau,\sigma\}\};\\
\R_3=&\{\overline \tau u =  \overline \tau, 
\overline \tau \overline \varrho^{m-1} u \overline \varrho = u\overline\varrho^{m-1}u\overline\varrho\overline 
\tau,\ \overline \tau \overline \varrho^{m-j+1} u \overline \varrho^{j-1} = \overline \varrho^{m-j+1} u \overline \varrho^{j-1} \overline \tau
\mid \\
&3\leq j\leq m,\, u\in\{\pi,\varrho,\tau,\sigma\}\}.
\end{align*}

The next result states presentations for $\PT_n\wr\T_m$ and for $\PT_{n\times m}$ 
with respect to the generating sets 
$\left\{\pi,\varrho,\tau,\sigma,\overline \pi,
\overline \varrho,\overline \tau\right\}$ and 
$\left\{\pi\varphi,\varrho\varphi,\tau\varphi,\sigma\varphi,
\overline \pi\varphi,\overline \varrho\varphi,
\overline \tau\varphi\right\}$, respectively, 
where $\varphi:\PT_n\wr\T_m\rightarrow \PT_{n\times m}$ 
is the epimorphism defined before Lemma~\ref{l21}. To simplify notation, 
we will omit ``$\varphi$'' from the 
description of the elements of $\PT_{n\times m}$, and hence
we will consider an element of $\PT_n\wr\T_m$ as 
an element of $\PT_{n\times m}$.

\begin{theorem}\label{presth}
(i) $\PT_n\wr\T_m=\left<\pi,\varrho,\tau,\sigma,\overline \pi,\overline \varrho,
\overline \tau\mid
\R_P\cup\R_T\cup\R_1\cup\R_2\cup\R_3\right>;$

(ii) $\PT_{n\times m}=\left<\pi,\varrho,\tau,\sigma,\overline \pi,\overline \varrho,
\overline \tau\mid
\R_P\cup\R_T\cup\R_1\cup\R_2\cup\R_3\cup\{(\varrho\sigma)^n=(\varrho\sigma)^n\overline\tau\}\right>$.
\end{theorem} 
\begin{proof}
(i) By the definition of the wreath product, 
\begin{equation}\label{wrdec}
\PT_n\wr\T_m=\left((\PT_n)\comp\varepsilon 1\times\cdots\times (\PT_n)\comp\varepsilon m\right)\rtimes(\T_m\proj\varepsilon).
\end{equation}
For $i\in\{1,\ldots,m\}$ set $\pi_i$, $\varrho_i$, $\tau_i$ and $\sigma_i$
as $(1\,2)\comp\varepsilon i$, 
$(1\,2\ldots n)\comp\varepsilon i$, $[2,2,3,4,\ldots,n]\comp\varepsilon i$, and $[\zt,2,3,\ldots,n]\comp\varepsilon i$, respectively. 
Set $\G_i=\{\pi_i, \varrho_i, \tau_i,\sigma_i\}$ and
$\proj\G= \{\overline\pi, \overline\varrho, \overline\tau\}$. 
Then, by~\eqref{wrdec}, 
the set $\G_1\cup\cdots\cup\G_m\cup\proj \G$ is a generating set
of $\PT_n\wr\T_m$ and combining~\eqref{semidir}, \eqref{dirpr} and
the definition of the left action of $\T_m$ on $(\PT_n)^m$ given
in Section~\ref{notations}, we obtain that 
\begin{equation}\label{pres1}
\PT_n\wr\T_m=\left<\G_1\cup\cdots\cup\G_m\cup\proj \G\mid \R_{1P}\cup\cdots\cup\R_{mP}\cup\R_T\cup\R\right>
\end{equation}
where, for $i\in\{1,\ldots,m\}$ the set $\R_{iP}$ is the set of relations
that we obtain from $\R_P$ by substituting $\pi_i$, $\varrho_i$, $\tau_i$ and 
$\sigma_i$ in the places of $\pi$, $\varrho$, $\tau$ and 
$\sigma$, respectively.
Moreover, $\R$ is the set consisting of the following relations:
\begin{equation}\label{eqc}
uv=vu\quad u\in\ \G_i,\ v\in \G_j\mbox{ with }i<j;
\end{equation}
\begin{align}
\label{eqp1}\overline\pi\pi_1&=\pi_2\overline\pi,\ &\overline\pi\varrho_1&=\varrho_2\overline\pi,\ &\overline\pi\tau_1&=\tau_2\overline\pi,\ &\overline\pi\sigma_1&=\sigma_2\overline\pi;&\\
\label{eqp2}\overline\pi\pi_2&=\pi_1\overline\pi,\ &\overline\pi\varrho_2&=\varrho_1\overline\pi,\ &\overline\pi\tau_2&=\tau_1\overline\pi,\ &\overline\pi\sigma_2&=\sigma_1\overline\pi;&\\
\label{eqp3}\overline\pi\pi_i&=\pi_i\overline\pi,\ &\overline\pi\varrho_i&=\varrho_i\overline\pi,\ &\overline\pi\tau_i&=\tau_i\overline\pi,\ &\overline\pi\sigma_i&=\sigma_i\overline\pi& 3\leq i\leq m;\\
\label{eqr1}\overline\varrho\pi_i&=\pi_{i-1}\overline\varrho,\ &\overline\varrho\varrho_i&=\varrho_{i-1}\overline\varrho,\ &\overline\varrho\tau_i&=\tau_{i-1}\overline\varrho,\ &\overline\varrho\sigma_i&=\sigma_{i-1}\overline\varrho& 2\leq i\leq m;\\
\label{eqr2}\overline\varrho\pi_1&=\pi_{m}\overline\varrho,\ &\overline\varrho\varrho_1&=\varrho_{m}\overline\varrho,\ &\overline\varrho\tau_1&=\tau_{m}\overline\varrho,\ &\overline\varrho\sigma_1&=\sigma_{m}\overline\varrho;&\\
\label{eqt1}\overline\tau\pi_1&=\overline\tau,\ &\overline\tau\varrho_1&=\overline\tau,\ &\overline\tau\tau_1&=\overline\tau,\ &\overline\tau\sigma_1&=\overline\tau;&\\
\label{eqt2}\overline\tau\pi_2&=\pi_1\pi_1\overline\tau,\ &\overline\tau\varrho_2&=\varrho_1\varrho_2\overline\tau,\ &\overline\tau\tau_2&=\tau_1\tau_2\overline\tau,\ &\overline\tau\sigma_1&=\sigma_1\sigma_2\overline\tau;&\\
\label{eqt3}\overline\tau\pi_i&=\pi_i\overline\tau,\ &\overline\tau\varrho_i&=\varrho_i\overline\tau,\ &\overline\tau\tau_i&=\tau_i\overline\tau,\ &\overline\tau\sigma_i&=\sigma_i\overline\tau&3\leq i\leq m.
\end{align}

Note that $\pi_1=\pi$, $\varrho_1=\varrho$, $\tau_1=\tau$, and $\sigma_1=\sigma$.
Further, $i\geq 2$, we have that $\pi_i=\proj\varrho^{m-i+1}\pi\proj\varrho^{i-1}$ and
we have analogous expressions for $\varrho_i$, $\tau_i$, and $\sigma_i$.
Hence for $i\geq 2$, we can remove the elements of $\G_i$ from the
generating set given in~\eqref{pres1} and we can also remove
the relation sets $\R_{iP}$.
Further we replace the occurrences in $\R$ 
of the
generators belonging to $\G_i$  by the corresponding expressions in $\pi$, 
$\varrho$, $\tau$ and $\sigma$. This way the relations 
in~\eqref{eqr1}--\eqref{eqr2} will transform into the trivial relation.
As $\pi^2=1$ follows from the relations
in $\R_P$, the relations in~\eqref{eqp2} can be removed, as they follow
from the relations in~\eqref{eqp1}.
This way,  the relations in~\eqref{eqc} lead to 
the relations in $\R_1$, the relations in~\eqref{eqp1} and~\eqref{eqp3} give
the relations in $\R_2$, and finally the relations 
in~\eqref{eqt1}--\eqref{eqt3} will result in  the relation set $\R_3$.
This shows that the presentation in part~(i) is valid.

(ii) As $(\varrho\sigma)^n=\zt$, 
the second statement follows from  part~(i) and Lemma~\ref{l21}(iii).
\end{proof}

Theorem~\ref{presth} gives a presentation of $\PT_n\wr\T_m$ and $\PT_{n\times m}$ 
in terms of 7 generators, even though we proved in Theorem~\ref{rankmain}
that these monoids are generated by 5~generators. 
In order to pass to the 5-element generating set, we  replace
the 4-generator set $\{\pi,\varrho,\overline\pi,\overline\varrho\}$
of $\S_n\wr\S_m$
used in Theorem~\ref{presth} with the generating set $\{\xi_1,\xi_2\}$
given at the beginning of Section~\ref{sectrank}.
In order two write down a presentation that is satisfied
by the generating set $\{\xi_1,\xi_2,\tau,\sigma,\overline\tau\}$,
we need to express the generators $\pi,\varrho,\overline\pi,\overline\varrho$
in terms of the generators $\xi_1$ and $\xi_2$.
We do this in the next lemma, whose proof, though rather technical, is routine, 
and so we omit the details. We note that the following lemma also 
implies~\cite[Theorem~4.1]{Araujo&Schneider:2009} that
$\S_n\wr\S_m$ is generated by two elements.

\begin{lemma}
Let $m, n\geq 2$. If $m$ and $n$ are both even then 
\begin{align}
\label{eq1}\pi&=(\xi_1^{m-1}\xi_2^2)^{(m-2)(n-1)^2} (\xi_1\xi_2^2)^{(m-1)(n^2-n-1)} (\xi_1\xi_2)^m; \\
\label{eq2}\varrho&=\left( (\xi_1^{m-1}\xi_2^2)^{(m-2)(n-1)^2} (\xi_1\xi_2^2)^{(m-1)(n^2-n-1)} \right)^{n-1}; \\
\label{eq3}\overline\pi&=\xi_2^{2n-1} \left( (\xi_1^{m-1}\xi_2^2)^{(m-2)(n-1)^2} (\xi_1\xi_2^2)^{(m-1)(n^2-n-1)} \right)^{n-1};\\
\label{eq4}\overline\varrho&=\xi_2^{2n-1} \left( (\xi_1^{m-1}\xi_2^2)^{(m-2)(n-1)^2} (\xi_1\xi_2^2)^{(m-1)(n^2-n-1)} \right)^n (\xi_1\xi_2)^m \xi_2^{2n-1} \xi_1 \xi_2. 
\end{align}
If either $m$ or $n$ is odd then 
\begin{align}
\label{eq5}\pi&=(\xi_1\xi_2^{2n-1})^{(m-1)n} (\xi_1^{m+1}\xi_2^{2n-1})^{(m-1)(n-1)} \xi_1^m;\\
\label{eq6}\varrho&=(\xi_1\xi_2^{2n-1})^{n-1} (\xi_1^{m+1}\xi_2^{2n-1})^{(m-2)(n-1)};\\
\label{eq7}\overline \pi&=\xi_1^{2m-1} (\xi_1\xi_2^{2n-1})^n (\xi_1^{m+1}\xi_2^{2n-1})^{(m-2)(n-1)};\\
\label{eq8}\overline \varrho&=\xi_1^{2m-1}  (\xi_1\xi_2^{2n-1})^{(m-1)n} (\xi_1^{m+1}\xi_2^{2n-1})^{(m-1)(n-1)} \xi_1^{m+2}.
\end{align}
\end{lemma}

Suppose that $w_1$, $w_2$, $w_3$ and $w_4$ are the words on the right hand sides 
of the equations~\eqref{eq1}--\eqref{eq4}, when both $m$ and $n$ are even,
or of the equations~\eqref{eq5}--\eqref{eq8} otherwise, respectively.
Let $\overline \R_P$, $\overline \R_T$, $\overline \R_1$, $\overline 
\R_2$ and $\overline \R_3$
the sets of relations that are constructed by 
substituting in $\R_P$, $\R_T$, $\R_1$, $\R_2$ and $\R_3$
the words $w_1$, $w_2$, $w_3$, $w_4$ in the places of $\pi$, $\varrho$, 
$\overline\pi$, $\overline\varrho$, respectively. 
Further, let $\overline r$ be the relation obtained by 
substituting $w_2$  in the relation 
$(\varrho\sigma)^n=(\varrho\sigma)^n\overline\tau$ 
into the place of $\varrho$. Then we obtain 
the following main result of this section.

\begin{corollary}\label{relcor}
(i) $\PT_n\wr\T_m=\left<\xi_1,\xi_2,\tau,\sigma,\overline \tau\mid \overline \R_P\cup\overline\R_T\cup\overline \R_1\cup\overline\R_2\cup\overline \R_3\right>$;

(ii) $\PT_{n\times m}=\left<\xi_1,\xi_2,\tau,\sigma,\overline \tau\mid \overline \R_P\cup\overline\R_T\cup\overline \R_1\cup\overline\R_2\cup\overline \R_3\cup\{\overline r\}\right>$.
\end{corollary}
\section*{Acknowledgement}

Cical\`o gratefully acknowledges the support of FCT and PIDDAC, within the
project PTDC/MAT/69514/2006 of CAUL and the support of the Regione Sardegna,
within the project Master \& Back, PR-MAB-A2009-837. She also thanks the
Department of Mathematics of the University of Trento for its hospitality from
2011.
Fernandes  gratefully
acknowledges support of FCT and PIDDAC, within the projects ISFL-1-143 and
PTDC/MAT/69514/2006 of CAUL.
Schneider was supported by the FCT project PTDC/MAT/101993/2008.


\end{document}